\documentclass{elsarticle}               
                             
\journal{Information Processing Letters}

\usepackage{\jobname}
\usepackage[linesnumbered,noline,ruled]{algorithm2e}

\newcommand{\rememberlines}{\xdef\rememberedlines{\number\value{AlgoLine}}}
\newcommand{\resumenumbering}{\setcounter{AlgoLine}{\rememberedlines}}

\begin{document}

\begin{frontmatter}
  \title{Verification and generation of unrefinable partitions}

  \author{Riccardo Aragona}
  \ead{riccardo.aragona@univaq.it}
  
  \author{Lorenzo Campioni}
  \ead{lorenzo.campioni1@graduate.univaq.it}
  
  \author{Roberto Civino}
  \ead{roberto.civino@univaq.it}
  \address{Dipartimento di Ingegneria e Scienze dell'Informazione e Matematica - Universit\`{a} dell'Aquila, Italy}  
  
  \author{Massimo Lauria}
  \ead{massimo.lauria@uniroma1.it}
  \address{Dipartimento di Scienze Statistiche - Sapienza Universit\`{a} di Roma, Italy}

  \begin{keyword}
    integer partitions into distinct part \sep minimal excludant \sep algorithms.
  \end{keyword}

\begin{abstract}
Unrefinable partitions are a subset of partitions into distinct parts which 
satisfy an additional unrefinability property. More precisely, being  an unrefinable partition means that none of the parts can be written as the sum of smaller integers 
without introducing a repetition.
 We address the algorithmic aspects of unrefinable partitions, such as testing whether  a given partition is unrefinable or not and 
 enumerating all the partitions whose sum is a given integer. We design two algorithms to solve the two mentioned problems and we discuss their complexity. 
\end{abstract}

\end{frontmatter}

\section{Introduction}
\label{sec:preliminaries}

Given $N \in \N$, a \introduceterm{partition} of $N$ is a finite sequence of 
positive integers $\aname_{1}, \aname_{2}, \dots, \aname_{k}$ such that
\begin{equation*}
  N = \aname_{1} + \aname_{2} + \cdots + \aname_{k}.
\end{equation*}
Each $\aname_{i}$ is called a \introduceterm{part} of the partition, 
and a partition into \introduceterm{distinct parts} is a partition
where all parts are distinct integers.
In this work we only discuss partitions into distinct parts and, among them, we
are  interested in the study of \introduceterm{unrefinable}
partitions,
where refining a partition means splitting any of its parts as the sum of smaller
pieces.
When the repetition of elements is allowed, the process of refining
a partition can proceed until one obtains
\[
  N = \underbrace{1 + 1 + \cdots +1}_{N \text{\ times\ }}
\] and for this
reason it is of little interest.

The problem becomes more interesting if we require the considered partitions, after the refinement, to still be  partitions into
distinct parts. Consider for example the following partition of $50$:
\[
  50 = 1 + 2 + 3 + 4 + 6 + 7 + 11 + 16.
\]
Is there a way to refine any
of the parts without introducing a repetition? In this case the answer is no, but how can this be tested and what effort is required?

There are well known algorithms and formulas to enumerate or count all
partitions of a given $N$~\cite{andrews1998theory}, but none is known, to our knowledge, for unrefinable partitions.
In order to close this gap,  we address in this paper some algorithmic aspects related to unrefinable partitions. We present two algorithms: one
which verifies whether a partition is unrefinable or not, and one  
  which recursively enumerates all unrefinable partition of a given $N$.
More precisely,
we  discuss a na\"ive $O(\ell^3)$-algorithm which determines if an increasing sequence of integers with maximum element equal to $\ell$  
represents an unrefinable partition and show that it can be improved by means of simple arithmetical arguments. We prove that the verification problem can be actually solved 
 in $O(\alength + \mu^{2})$ steps, where $\mu$ is the \emph{minimal excludant} of the sequence, i.e.\ the least integer that is not a part, which is defined in detail in the next section.
 Following Aragona et al.~\cite{aragona2021maximal} we have that $\ell$ and $\mu$
 are upper bounded by some functions in \(O(\sqrt{N})\), therefore the
 verification algorithm is linear in $N$ in the worst case.

\subsection{Related works}
The notion of unrefinable partition is at least as old as the OEIS entry A179009~\cite{OEIS} (due to David S.\ Newman in 2011) and has been formally introduced in a paper by Aragona et al.~\cite{aragona2021unrefinable}, where unrefinability 
appeared in a natural way in connection to some subgroups in a chain of normalizers~\cite{aragona2021rigid}. The authors proved that 
the generators of such subgroups are parametrized by some unrefinable partitions satisfying additional conditions on the minimal excludant.
Some first combinatorial equalities regarding unrefinable partitions for triangular and non-triangular numbers have been shown recently~\cite{aragona2021maximal,aragona2022number}. 
The notion of minimal excludant, which frequently appears also in combinatorial game theory~\cite{gurvich2012further,fraenkel2015harnessing}, has  been studied in the context of integers partition by other authors~\cite{andrews2019,Ballantine2020,Hopkins2022}.

\subsection{Organization of the paper}
The remainder of this document is arranged as follows. In Sec.~\ref{sec:nota} we introduce our notation and show some preliminary results. In Sec~\ref{sec:refinability} we present the verification algorithm (cf.~Algorithm~\ref{alg:simple}), prove its correctness and discuss its complexity (cf.\ Theorem~\ref{thm:main}). The enumerating algorithm (cf.\ Algorithms~\ref{alg:enum} and \ref{alg:final}) and the relative complexity analysis (cf.~Theorem~\ref{thm:enumeration}) are presented in~Sec~\ref{sec:enum}, which concludes the paper.

\section{Notation and preliminaries}
\label{sec:nota}
In this paper we use a non-conventional representation for
partitions into distinct parts.
Namely, together with the integers which belongs to the partition, we explicitly mark
all the missing parts up to a certain value.
Formally, we call
a \introduceterm{\ps} a sequence of integers
$\lambda = (v_i)_{i\geq 1}$ such that
\begin{enumerate}
\item \label{three} there exists  $\ell \geq 1$ such that  for $i > \ell$ we have $v_i = 0$,
\item  for each $1 \leq i \leq \ell$ we have that $\apart_{i}$ is either $i$ or $0$.
\end{enumerate}
Each null $v_i$ is displayed using the symbol $\star$,
and we denote for brevity $\lambda =(\aseq)$, where $\ell$ is as above, and $i$ is called the \emph{index} of $v_i$.  
Defining 
\[
\asum(\aname) = \sum_{i=1}^{\infty} v_i=  \sum_{\apart_{i}=i} v_i < \infty,
\] where the symbol $\star$ is treated as zero, any such
sequence naturally represents a partition of the integer $N=\asum(\aname)$ into distinct parts. Therefore,
when it is not ambiguous, we use the terms partition and \ps interchangeably.  
We call \emph{length} of  the \ps the integer $ \ell =|\aname|$.
If $x \leq |\aname|$ and $\apart_{x}$ is equal to $x$ then we say that
$x$ is a part of $\aname$, and say that $x \in \aname$. Otherwise we
say that $x \not\in \aname$.

The representation of a partition as a \ps is clearly not unique and this is exactly the reason why we need to introduce this formalism.
We use this representation to distinguish, e.g., that
$(1,2,\star,\star,5)$ is a prefix of
$(1,2,\star,\star,5,6,7,\star,9)$ while $(1,2,\star,\star,5,\star)$ is
not, even though they both represent the partition of $8$ with parts $1$,
$2$, and $5$.
This distinction will be useful while discussing the process of 
verifying  and enumerating partitions. Indeed, our algorithms will take \pss as inputs and will
\begin{enumerate}
  \item test whether a \ps is unrefinable,
  \item enumerate all unrefinable \pss of a given $N$ which does not end with $\star$.
\end{enumerate}

Our representation highlights the numbers that are not included among
the parts of the partitioned number $N$. We say that $x$ is
a \introduceterm{missing part} of $\aname$ when $x \leq |\aname|$ and the
$x$th coordinate of $\aname$ is  $\star$.\footnote{%
  Notice that an integer $x \not\in \aname$ is not called a missing
  part when $x > |\aname|$. }
The smallest missing part in $\aname$
(\ie the index of the leftmost $\star$ in the sequence) is called the
\introduceterm{minimal excludant} of $\aname$, and is denoted by
$\mex(\aname)$.
It is customary in the relevant literature to set $\mex(\aname)=0$ when there is no
such minimal excludant, \ie when $\aname$ does not have any $\star$.

For conveniency, we abuse notation and we define
$\aname \cup \{\star\}$ as the \ps obtained concatenating a $\star$
to $\aname$. In the same way, we denote $\aname \,\cup \{x\}$ as the
concatenation of $\aname$ with number $x$, the latter operation being well
defined only when $x=|\aname|+1$.

Let us now define the notion of refinability in the context of \pss.
\begin{definition}[{Refining}]
  \label{def:refine}
  Let us consider a \ps
  \begin{equation*}
    \aname = (\aseq).
  \end{equation*}
  A part $\apart_{r}$ is \introduceterm{refinable} when
  the equation $r = r_{1} + \cdots + r_{t}$ holds for some $t>1$,
  with $\apart_{r} = r$ and $\apart_{r_{1}}, \ldots, \apart_{r_{t}}$
  all equal to $\star$.
  Accordingly, the equation $r = r_{1} + \cdots + r_{t}$ is 
  a \introduceterm{refinement} of $\apart_{r}$.
  A part that has no refinement is called
  \introduceterm{unrefinable}.
  A partition $\aname$ is \introduceterm{refinable} if some of its
  part admits a refinement, and it is
  \introduceterm{unrefinable} otherwise.
\end{definition}

Since we are discussing algorithmic matters, it is a legitimate concern
whether the length of a \ps might be longer than the length of a simple list
of parts (missing parts excluded). %
However it turns out that, if representing unrefinable partitions, their sizes are actually similar, as  
 the following  known lemma~\cite{aragona2021maximal} and its corollary show.

\begin{lemma}[\cite{aragona2021maximal}]
  Let $\aname=(\aseq)$ be a \ps with $\alast=\alength$ representing an unrefinable partition.
  Then the number of missing parts in $\aname$ is at most
  $\lfloor {\alength}/{2} \rfloor$.
\end{lemma}

\begin{corollary}
  \label{cor:maximal}
  Let $\aname=(\aseq)$, with
  $\alast=\alength$, be  a \ps representing an unrefinable partition and let $k$ be the number of parts in $\aname$. Then
  \begin{equation*}
    k \leq \alength \leq 2k.
  \end{equation*}
\end{corollary}
\begin{proof}
  Let $k$ and $m$ be respectively the number of parts and missing parts
  in $\aname$.
  A \ps of length $\alength$ can have at most $\alength$ parts,
  therefore $k \leq \alength$.
  For the second inequality we have that $m$ is at most
  $\lfloor {\alength}/{2} \rfloor$ by the previous
  lemma, and furthermore that $\alength=m+k$. Hence
  $\left\lceil \alength/2 \right\rceil \leq k$, which implies
  $\ell \leq 2k$.
\end{proof}

At first glance it may seem that checking  refinability for a partition $\aname$
should be computationally expensive, since, according to Definition~\ref{def:refine}, we
potentially need to check all sums of two or more indexes that
corresponds to $\star$ in $\aname$.
However, it is not hard to  realize that if a $\aname$
is refinable, then there exists a part $r$ with some refinement of the form
$r = a + b$.

\begin{proposition}\label{proposition:simplerefinement}
  If a partition $\aname$ has some refinement, then its smallest
  refinable part $r$ has a refinement of the form $r = a + b$.
\end{proposition}
\begin{proof}
  Let $r$ be the smallest refinable part for which there exists some refinement
  $r = \nu_{1} + \cdots + \nu_{t}$. If $t=2$ there is nothing to prove. Otherwise, let us
  fix $a=\nu_{1}$ and $b=\nu_{2} + \cdots + \nu_{t}$.
  If $b \in \aname$, then $b=\nu_{2} + \cdots + \nu_{t}$ would
  be a refinement itself, but $b < r$ and this would violate
  the minimality of $r$, hence $b$ is not a part of $\aname$.
  This shows that $r = a + b$ is indeed a refinement of
  $\aname$.
\end{proof}

From Proposition~\ref{proposition:simplerefinement}
it is easy to obtain a polynomial algorithm that checks refinability of a \ps
$\aname$ in time $O(\alength^{3})$: for every part of $\aname$, test
whether it is the sum of two smaller missing parts.
We will show in the next section how this algorithm can be improved.

\section{A faster algorithm to check refinability}
\label{sec:refinability}

In this section we introduce Algorithm~\ref{alg:simple} to check
refinability of \pss. This algorithm is faster compared to the na\"ive
one discussed above.
Our improvement comes from the key observation that whenever $a$ and
$b$ are missing elements in an unrefinable partition $\aname$, then
for any two integers $x,y >0$ we have that $x a + y b$ cannot be
in $\aname$.
This leads to the following idea: once we know $\mu=\mex(\aname)$, if
we find out that $x$ is another missing part, then none of $x+\mu, x +2\mu, x+3\mu,
\ldots$ can be parts of $\aname$, unless $\aname$ is refinable.
Hence, for each $0 \leq j < \mu$ we just need to keep trace of the
first missing part (greater than $\mu$) that is equal to
$j \pmod{\mu}$ to completely characterize all parts that would
violate unrefinability.

\begin{example}
Assume we want to check whether a sequence \(\aname\) is refinable and
that $\aname$ has \((1,2,3,\star,5,6,\star,8,9, \star)\) as prefix.
We know that \(\mex(\aname)=4\), and since \(7 \not\in \aname\),
the numbers \(\{11,15,19,\ldots\}\) cannot be parts of \(\aname\) unless it
is refinable. Similarly, integers as \(\{14,18,22,\ldots\}\) are forbidden
as well because \(10 \not\in \aname\).
The integer \(17\) is forbidden too, because it is \(7 + 10\), and
therefore also \(\{21,25,29, \ldots\}\) are
forbidden.
Essentially, after reading the first \(10\) elements of \(\aname\),
we already know that \(7+4t\) cannot be in \(\aname\) for all
\(t \geq 1\),
unless the sequence
is refinable, i.e.\ that every integer larger than $7$ which belongs to the residue class
of 3 modulo $\mu=4$ would violate unrefinability.
The same for \( 10+4t \) for \(t \geq 1\) and for \(17+4t\) for \(t \geq 0\).
The distinction between the case  \(t \geq 0\) and  \(t \geq 1\) will be clear in the next paragraph.
\end{example}

The algorithm that we are about to present scans the values of $\aseq$
from index $\mu+1$ to $\alength$, while maintaining the
information about which numbers in each residue class modulo $\mu$ are
\introduceterm{forbidden}, \ie are either known missing parts or
violate unrefinability if met later when scanning the sequence.
 
More precisely, this is accomplished by  defining 
$\mu$ counters $p_j$ for each residue class $0 \leq j < \mu$ modulo $\mu$. We initially set 
  $p_{j}:=\infty$, meaning  there is no forbidden number in the
residue class $j$.
Going from $\mu+1$ to $\alength$, the values of each $p_{j}$ is
updated every time a missing part is met in the \ps.
The invariant is when we reach position $\apart_{t}$ in the sequence,
the values
\begin{equation*}
p_{j} \quad p_{j} + \mu \quad p_{j} + 2\mu \quad p_{j} + 3\mu \quad \ldots  
\end{equation*}
are all forbidden in any unrefinable \ps starting with
the same prefix $\apart_{1}, \ldots, \apart_{t}$.
This is indeed enough: we can update $p_{j}$s so that when
the scan reaches $\alast$ without meeting any forbidden value, then
$\aname$ is unrefinable.
  Algorithm~\ref{alg:simple} proceeds as follows. It starts by finding the minimal excludant $\mu$,
  and 
  if none exists, then the partition is obviously unrefinable.
  Then it checks all integers from index $\mu+1$ to $\alast$ in order.
  For a given number $r$ in this sequence there are two possibilities:
  \begin{itemize}
    \item $r \in \aname$ and therefore we need to check if it has
    a refinement;
    \item $r \not\in \aname$ and then we update our knowledge of which
    numbers would contradict refinability, if met.
  \end{itemize}
  In the second case, such knowledge is represented by the numbers
  $p_{0}, \ldots, p_{\mu-1}$ which are updated at each iteration of
  the main loop of the algorithm. If $p_{j}$ is finite, then it is
  equal to $j$ modulo \(\mu\).
  
  \begin{algorithm}
\DontPrintSemicolon
\SetKwInOut{Input}{Input}\SetKwInOut{Output}{Returns}
\SetKw{Continue}{continue}
\Input{$\aname=(\aseq)$}
\Output{{\sc Refinable} or {\sc Unrefinable}}
\BlankLine
$\mu \leftarrow \mex(\aname)$\;
\lIf{$\mu = 0$} {\Return {\sc Unrefinable}}
$\vec{p}=(p_0, \ldots, p_{\mu-1}) \leftarrow (\infty, \infty,\ldots,\infty)$\;
\For{$r$ in $(\mu+1), \ldots, \alength$}{\label{alg:mainloop}%
  $j \leftarrow r \pmod{\mu}$\;

  \lIf{$v_{r} = r$ and $v_r \geq p_{j}$}
  {\Return {\sc Refinable}\label{alg:returnRef}}

  \lIf{$v_{r} = \star\quad$}
  {$\vec{p} \leftarrow$\textsc{Update}$(\vec{p},r)$}
}
\Return {\sc Unrefinable}
\caption{\label{alg:simple}\textsc{Verify} (an algorithm to check refinability)}
\rememberlines
\end{algorithm}

  \begin{algorithm}
  \resumenumbering
\DontPrintSemicolon
\SetKwInOut{Input}{Input}\SetKwInOut{Output}{Returns}
\SetKw{Continue}{continue}
\Input{$\vec{p}=(p_0,\ldots,p_{\mu-1})$, $r$ a newly discovered
  missing part}
\Output{$\vec{p}=(p_0,\ldots,p_{\mu-1})$, updated}
\BlankLine
$j \leftarrow r \pmod{\mu}$\;
\eIf{$r > p_{j}$\label{alg:startupdateiter}}{\label{alg:test1}%
  $t \leftarrow r + p_{j} \pmod{\mu}$\;
  $p_{t} \leftarrow \min(p_{t},r+p_{j})$\label{alg:setpt1}\;
}{%
$p_j \leftarrow r$\label{alg:setpj}\;
\For{$j'$ in $\{1, \ldots, \mu-1 \}\setminus{\{j\}}$}{%
$t \leftarrow j + j' \pmod{\mu}$\;
$p_{t} \leftarrow \min(p_{t},p_{j} + p_{j'})$\label{alg:setpt2}\;
}
}
\Return {$(p_0, \ldots, p_{\mu-1})$}
\caption{\label{alg:update}\textsc{Update} (improves $p_{j}$s after
  a new missing part $r$ is discovered)}
\end{algorithm}

    Algorithm~\ref{alg:simple} uses a subroutine called \textsc{Update} (cf.~Algorithm~\ref{alg:update}).
     Once we discover a new missing part
  \(r\), two different circumstances can occur, and they are addressed accordingly by \textsc{Update}. Precisely,  either \(r\) is
  not the smallest missing part in its residue class, and in this case
  we just need to see how \(r\) interacts with smaller missing parts
  in the same residue class (\textbf{if} branch);  or \(r\) is the
  smallest missing part in its residue class, and then we need to
  check how this influences all other missing parts (\textbf{else} branch). We give examples
  for both cases, precisely Example~\ref{exa:else} for the \textbf{else} branch 
  and Example~\ref{exa:if} for the \textbf{if} branch.
  
   \begin{example}\label{exa:else}
    Let
    \(\aname = (1,2,3,4,5,\star,7,8,\star,\star,11,12,13,\star,
    \ldots)\) and consider all calls to \textsc{Update} when
    Algorithm~\ref{alg:simple} reaches \(14\) in $\aname$.
    We have \(\mu=6\).
    The first call sets \(p_{3}=9\). The second call sets
    \(p_{4}=10\), and since \(19=9+10\), we need \(19\) to be forbidden as
    well. This happens in the \textbf{for} loop that sets
    \(p_{1}=19\). At this stage we have $(p_0,p_1,p_2,p_3,p_4,p_5) = (\infty,19,\infty,9,10,\infty)$.
    The third call happens when the scan reaches \(14\). Here the algorithm sets
    \(p_{2}=14\), and afterward the \textbf{for} loop in
    line~\ref{alg:setpj} computes the forbidden values
    \[
      19 + 14 = 33, \quad 9 + 14 = 23, \quad 10 + 14 = 24.
    \]
    The information that \(33\) is forbidden is included in \(p_{3}\) 
    (previously set to \(9\)), while the information \(p_{5}=23\) and
    \(p_{0}=24\) is newly determined.
    When $14$ is reached and processed, the information on forbidden numbers is represented by \[(p_0,p_1,p_2,p_3,p_4,p_5) = (24,19,14,9,10,23).\]
    
    The partition $\lambda$ may continue either with $15$ or with $\star$. Notice that $15 = 3 \pmod 6$ and $15>p_3=9$, therefore $15 \in \lambda$ would 
    prove refinability (indeed $15=6+9$). Therefore $\lambda$ can only continue with $\star$.
  \end{example}

  \begin{example}\label{exa:if}
    Let
    \(\aname = (1,2,3,\star,5,6,\star,8,9, \star,\star,12,13,\star,
    \ldots)\) and consider the call to \textsc{Update} when
    Algorithm~\ref{alg:simple} reaches \(14\) in
    \(\aname\). We have \(\mu=4\).
    By the time the algorithm scans position \(14\) we know that the
    sequence misses parts \(10\) and \(14\), therefore \(24\) must be
    forbidden as well. 
    Indeed in this call we have \(r=14\) and \(p_{2}=10\), and
    line~\ref{alg:setpt1} runs and sets \(p_{0}\) to \(24\) as desired.
  \end{example}

  In the rest of the section we discuss the correctness and complexity
  of Algorithm~\ref{alg:simple}. First we prove the correctness of the
  algorithm on unrefinable and refinable \pss separately, then we
  discuss its complexity.

  \begin{lemma}
    \label{lmm:corr_unrefinable}
    Algorithm~\ref{alg:simple} outputs \textsc{Unrefinable} on every
    unrefinable \(\aname\).
  \end{lemma}
  \begin{proof}
    Consider an unrefinable \pss $\aname$.
    We start by proving that when the algorithm assigns a value \(w\) to
    some \(p_{j}\), it means that $w \not\in \aname$.
    We prove this by induction on the iterations of the main loop at
    line~\ref{alg:mainloop}.
    The base of the induction trivially holds because before the loop
    all $p_{j}$ are set to $\infty$.

    For the inductive step we discuss all the ways these assignments
    occur in the \textsc{Update} function described in
    Algorithm~\ref{alg:update}.
    If we set $p_{j}$ to $w$ at line~\ref{alg:setpj}, then
    $w\not\in\aname$ because \textsc{Update} would have been called
    when $\apart_{w}=\star$.
  If we update $p_{t}$ to value $w$ either at line~\ref{alg:setpt1} or
  at line~\ref{alg:setpt2}, we already know that $\apart_{r} = \star$
  and, by induction, that $p_{j}$ and $p_{j'}$ are not in
  $\aname$.
  Since $\aname$ is unrefinable, the new value of $p_{t}$ (namely $w$)
  cannot be in $\aname$ either.

  We just proved that, at any moment in the algorithm, every finite
  valued \(p_{j}\) is not in \(\aname\). We improve this by showing that
  the same holds for \(p_{j} + t\mu\) for $t \geq 0$, by induction on
  \(t\).
  The case \(t=0\) is what we have proved so far. Assuming
  \(p_{j} + t \mu \not\in\aname\), then by unrefinability the same
  holds for \(p_{j} + (t+1)\mu\). 
  
  To conclude, observe that the only possible way for
  Algorithm~\ref{alg:simple} to be incorrect is to return at
  line~\ref{alg:returnRef}. This happens when there is some
  $\apart_{r} = r$ which is greater than both $\mu$ and $p_{j}$, and
  that it is equal to $j$ modulo $\mu$. Hence $r = p_{j} + t\mu$ for
  some $t>0$. But we just showed that these values are not in
  \(\aname\), therefore the algorithm cannot return at
  line~\ref{alg:returnRef}.
  \end{proof}

  To prove the correctness of Algorithm~\ref{alg:simple} on refinable
  partition we use the following two propositions.

  \begin{proposition}\label{proposition:pastr}
    Consider the iteration $r$ of the main loop of
    Algorithm~\ref{alg:simple}, where $r \not \in \aname$ and
    $r=j \pmod{\mu}$. After that iteration, $p_{j} \leq r$.
  \end{proposition}
  \begin{proof}
    \textsc{Update} is called, and when it reaches
    line~\ref{alg:test1} either the test $r > p_{j}$ passes, or
    $p_{j}$ is set to $r$. Hence at the end of iteration $r$ we have
    that $p_{j} \leq r$. Successive iterations can only decrease the
    value of $p_{j}$.
  \end{proof}
  
  \begin{proposition}\label{proposition:setpj}
    Assume that Algorithm~\ref{alg:simple} reaches iteration \(r\),
    and let \(j\) be the residue class of \(r\) modulo \(\mu\).
    The assignment at line~\ref{alg:setpj} of \textsc{Update} is
    executed if and only if \(r\) is the smallest number strictly
    greater than $\mu$ in residue class \(j\) with $r\not\in\aname$.
  \end{proposition}
  \begin{proof}
    If there is a smaller \(r' \not\in \aname\) in the same residue
    class \(j\), then by proposition~\ref{proposition:pastr} we 
    have \(p_{j} \leq r' < r\). In that case, line~\ref{alg:setpj}
    is not reached.
       
    In the other direction, let \(r\) be the smallest number in the
    residue class \(j\) for which \(r\not\in\aname\).
    If $r \leq p_{j}$ at the time the main loop reaches iteration $r$,
    then  line~\ref{alg:setpj} is executed.
    Otherwise, \(p_{j}\) must have been assigned to the current value
    at lines~\ref{alg:setpt1} or~\ref{alg:setpt2} in some iteration
    $r'$ earlier than $r$. In both cases the assigned value is
    strictly larger than $r'$.
    Hence we have $r' < p_{j} < r$ and therefore $p_{j} \in \aname$ by
    hypothesis. Algorithm~\ref{alg:simple} returns {\sc Refinable} at
    iteration $p_{j}$ or earlier, and therefore never reaches
    iteration \(r\) as assumed.
  \end{proof}

  Now we can prove the correctness in the refinable case.

  \begin{lemma}
    \label{lmm:corr_refinable}
    Algorithm~\ref{alg:simple} outputs \textsc{Refinable} on every
    refinable \(\aname\).
  \end{lemma}
  \begin{proof}
  By Proposition~\ref{proposition:simplerefinement} we know that the smallest refinable
  part $r$ is refinable as $a+b$ with $a,b \not\in \aname$.
  Let us  denote  $j_{a}= a\pmod{\mu}$, $j_{b} = b \pmod{\mu}$, and
  $j_{r} = r \pmod{\mu}$. Clearly   $j_{r} = j_{a} + j_{b} \pmod{\mu}$.

  If the algorithm does not reach iteration $r$, it must be because it
  returned {\sc Refinable} earlier and so there is nothing to prove. Otherwise let us
  show that it must return {\sc Refinable} at iteration $r$.
  
  The case of $j_{a}=0$ is simple: we have that $j_{r}=j_{b}$ and
  $p_{j_{r}} \leq b$ by Proposition~\ref{proposition:pastr}. Therefore we get
  $r > b \geq p_{j_{r}}$ and the algorithm returns at
  line~\ref{alg:returnRef}. The case of $j_{b}=0$ is symmetric.

  For the remaining case of $j_{a}\neq 0$ and $j_{b}\neq 0$ we split
  into two further subcases: when $j_{a}=j_{b}$ and when
  $j_{a}\neq j_{b}$.

  When $j_{a}\neq 0$, $j_{b}\neq 0$ and $j_{a}=j_{b}$, we
  may assume without loss of generality that $a<b$.
  By the time the algorithm reaches iteration $b$, we have that
  $p_{j_{a}} \leq a$ because of Proposition~\ref{proposition:pastr}.
  The test at line~\ref{alg:test1} at call
  \textsc{Update}$(\vec{p},b)$ can be rewritten as $b > p_{j_{a}}$,
  hence the value $p_{j_{r}}$ is assigned to a number smaller or
  equal than $r = a+b$ in the residue class of $j_{r}$, in
  line~\ref{alg:setpt1}.
  At the time the main loop reaches the iteration $r$, the algorithm
  reaches line~\ref{alg:returnRef} and returns {\sc Refinable}.

  When $j_{a}\neq 0$ and $j_{b}\neq 0$ and $j_{a} \neq j_{b}$,
  we need to consider the smallest missing elements $a'$ and $b'$ 
  that are equal to $j_{a}$ and $j_{b}$, respectively, modulo $\mu$.
  We assume without loss of generality that $a' < b'$.
  When the algorithm reaches iteration $b'$ we have that
  $p_{j_{a}} \leq a'$ because of Proposition~\ref{proposition:pastr}, and that
  assignment $p_{j_{b}} \leftarrow b'$ in line~\ref{alg:setpj} is
  executed because of Proposition~\ref{proposition:setpj}.
  In the for loop right after line~\ref{alg:setpj}, we know that
  $j_{a} \in \{1,\ldots,\mu-1\}\setminus{j_{b}}$, therefore we get
  that $p_{j_{r}}$ is set to some value smaller or equal to $a'+b'$
  and in particular to a value smaller or equal than $r$. In the
  successive iteration the value never increases, and at iteration $r$
  we know that line~\ref{alg:returnRef} gets executed.
\end{proof}

\begin{lemma}[Running time]\label{lmm:runtime}
  Algorithm~\ref{alg:simple}, executed  on  the \ps \(\aname=(\aseq)\) with
  \(\mu=\mex(\aname)\), runs in time $O(\alength+\mu^{2})$.
\end{lemma}
\begin{proof}
  The initialization of the $p_{j}$s and the computation of
  $\mu=\mex(\aname)$ takes $O(\alength)$ steps.
  The main loop in Algorithm~\ref{alg:simple} is executed at most
  $\alength$ times. The inner loop in Algorithm~\ref{alg:update} is
  executed in at most $\mu$ of them.
  The total running time  is therefore $O(\alength+\mu^{2})$.  
\end{proof}

Putting together
Lemmas~\ref{lmm:corr_unrefinable},~\ref{lmm:corr_refinable}
and~\ref{lmm:runtime} we obtain the main theorem of this section.

\begin{theorem}
  \label{thm:main}
  Algorithm~\ref{alg:simple}, executed on the \ps $\aname=(\aseq)$, returns
  \textsc{Unrefinable} if and only if the partition $\aname$
  is unrefinable. Its complexity is $O(\alength + \mu^{2})$, where
  $\mu=\mex(\aname)$.
\end{theorem}

We conclude this section with some remarks on Algorithm~\ref{alg:simple},
which will be important in the next section regarding the enumerating algorithm.
First of all,
notice that if its main loop reaches an iteration $r$ where
$p_{j} \leq r$ for all $0 \leq j < \mu$, then the only possible way to
extend $\aname$ is adding an arbitrary number of \(\star\).
Any additional part would lead Algorithm~\ref{alg:simple} to return
\textsc{Refinable}. Moreover:

\begin{definition}\label{def:saturated}
  Let $\aname=(\aseq)$ be an unrefinable \ps  with
  $\mex(\aname)=\mu$ and let $\vec{p}$ be the values computed by
  Algorithm~\ref{alg:simple} on $\aname$. We say that $\aname$ is
  \introduceterm{saturated} when
  \begin{equation*}
   |\{p_{j} \leq \ell \;:\; 0 \leq j < \mu \}|=\mu.
  \end{equation*}
\end{definition}

\begin{proposition}\label{proposition:atmostmu}
  Consider an unrefinable partition $N=\aname_{1} + \aname_{2} +
  \ldots + \aname_{k}$ with  minimum excludant $\mu$.
  There are at most $\mu$ unrefinable \ps whose sum is $N$, which are
  not saturated, and which have parts $\{\aname_{1}, \aname_{2}, \ldots, \aname_{k}\}$.
\end{proposition}
\begin{proof}
  Assume without loss of generality that $\aname_{k}$ is the
  largest part.
  There is a unique \ps $\aname$ with parts
  $\{\aname_{1}, \aname_{2}, \ldots, \aname_{k}\}$ and length
  $\aname_{k}$.
  Any other partition is obtained adding stars to the sequence, but
  after adding $\mu$ stars the resulting \ps must
  be saturated.
\end{proof}

\section{How to enumerate unrefinable partitions}
\label{sec:enum}
The verification via Algorithm~\ref{alg:simple} of a \ps
$\aname=(\aseq)$ with $\mu=\mex(\aname)$ starts by scanning the interval
$\mu+1, \ldots, \alength$.
Up to the point when some index \(r\) is under scrutiny, the
algorithm uses no information about the elements of $\aname$ of successive
indexes. More concretely, the values $\vec{p}$ computed at iteration
$r$ are completely determined by the same old values computed at
iteration $r-1$ and by the fact that $r$ is either in $\aname$ or not.
Therefore we can design the enumeration process as the visit of the
tree of all possible \pss, so that the verification algorithm is run
on the sequence corresponding to any branch of tree (see
Figure~\ref{fig:tree}).
A branch is pruned as soon as the the corresponding sequence has no
possible extensions that are unrefinable and of sum at most \(N\).
When the sum of a sequence corresponding to a surviving branch equals the goal
value \(N\), the sequence is returned as output.

It is convenient to enumerate separately all unrefinable partitions of
\(N\) that have the same minimal excludant. Given $N$, we set $n$ as
the largest positive integer such that
\begin{equation*}
  \sum^{n}_{i=1} i \leq N 
\end{equation*}
and then we partition the search space of \pss according to prefixes:
\begin{equation}\label{eq:prefixes}
\begin{array}{ll}
  \aname^{\dagger} &=(1,2,3,4,\ldots,n-2,n-1,n),\\
  \aname^{n}&=(1,2,3,4,\ldots,n-2,n-1,\star),\\
  \aname^{n-1}&=(1,2,3,4,\ldots,n-2,\star),\\
  \vdots\\
  \aname^{4}&=(1,2,3,\star), \\
  \aname^{3}&=(1,2,\star),\\
  \aname^{2}&=(1,\star),\\
  \aname^{1}&=(\star).
\end{array}
\end{equation}

If $N$ is triangular, i.e.\ if $N={n(n+1)}/{2}$, then the
sequence
$\aname^{\dagger}$ itself is the unique unrefinable partition of \(N\)
with no minimal excludant, and it must be in the output of the enumeration.
If \(N\) is not triangular, i.e.\ if ${n(n+1)}/{2} < N$, there is no
unrefinable partition with prefix $\aname^{\dagger}$:  any additional
part would make the sum exceed \(N\).

Any other unrefinable partition of \(N\) must have minimal excludant
\(1 \leq \mu \leq n\), and for a given value of \(\mu\) there is a one-to-one
correspondence between these partitions and the \pss $\aname$
that
\begin{itemize}
  \item are unrefinable,
  \item have $\asum(\aname)=N$,
  \item have prefix $\aname^{\mu}=(1,2,3,4,\ldots,\mu-1,\star)$,
  \item have $\alast=\alength$ (i.e.\ not ending with $\star$).
\end{itemize}
In order to enumerate them, we describe the recursive algorithm
\textsc{Enumerate} (cf.~Algorihtm~\ref{alg:enum}).

\begin{algorithm}
\DontPrintSemicolon
\SetKwInOut{Input}{Input}\SetKwInOut{Output}{Output\ }
\SetKw{Continue}{continue}
\SetKw{Yield}{output}
\Input{$N$\\$\aname=(1,2,3,\ldots,\mu-1,\star,\apart_{\mu+1},
  \apart_{\mu+2}, \ldots)$,
  unrefinable\\$\vec{p}=(p_0,\ldots,p_{\mu-1})$}
\Output{All unrefinable partitions of $N$ with prefix $\aname$, not
  ending with $\star$.}
\BlankLine
$r \leftarrow |\aname|+1$\;
$j \leftarrow r \pmod{\mu}$\;
\BlankLine
\tcp*[l]{Cases when we extend with $r$, if possible}
\lIf{$r < p_{j}$ and $\asum(\aname)+r=N$\label{alg:output}}
{\Yield $\aname \cup \{r\}$}
\lIf{$r < p_{j}$ and $\asum(\aname)+r<N$\label{alg:keepgoing}}
{\textsc{Enumerate}($N$,\ $\aname \cup \{r\}$,$\ \vec{p}\ $)}
\BlankLine
\tcp*[l]{Case when we extend with $\star$}
$\vec{p} \leftarrow$\textsc{Update}$(\vec{p},r)$\;
\lIf{$\aname \cup \{\star\}$ is not saturated}{\textsc{Enumerate}($N$,\ $\aname \cup \{\star\}$,$\ \vec{p}\ $)}
\caption{\label{alg:enum} \textsc{Enumerate}}
\end{algorithm}

\textsc{Enumerate} starts with a \ps $\aname$ with
\(\mex(\aname)=\mu\) and extends it in all possible ways in a binary
tree-like fashion (cf.\ Figure~\ref{fig:tree}).
\begin{figure}
  \begin{center}
    \includegraphics[scale=0.8]{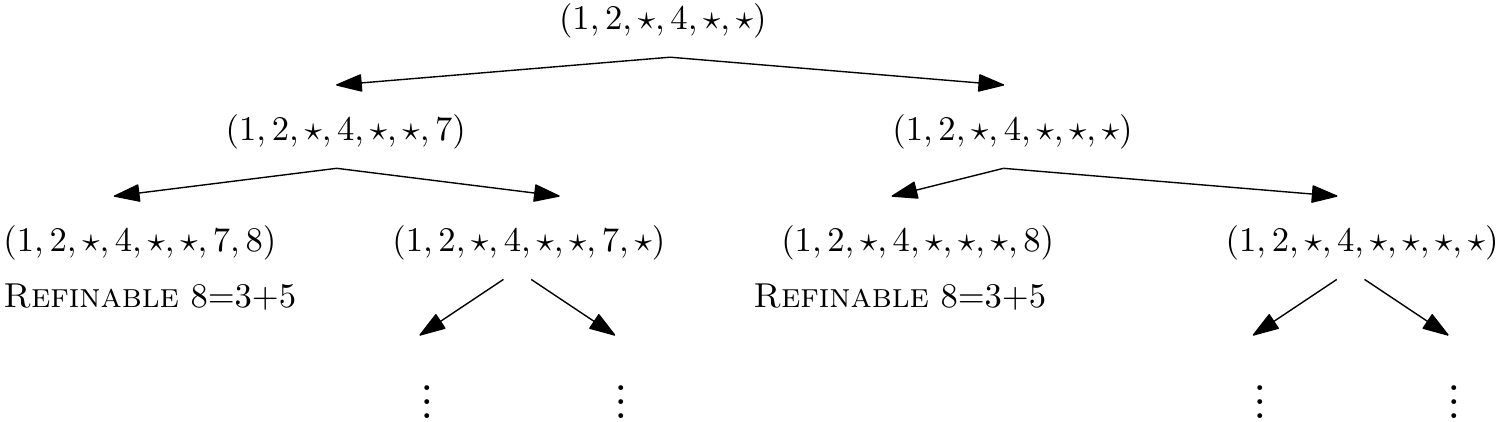}
  \end{center}
  \caption{The branching from the sequence
    $(1,2,\star,4,\star,\star)$. For any two sequences in the tree,
    the running of Algorithm~\ref{alg:simple} proceeds identically up
    to the point that the corresponding branches diverge.}
  \label{fig:tree}
\end{figure}
When visiting the node of the tree corresponding to sequence $\aname$,
the algorithm decides whether to branch on $\aname \cup \{v\}$, and
successively whether to branch on $\aname \cup \{\star\}$. Therefore,
the tree is visited in lexicographic order.
A branch is pruned either when a partition of $N$ is reached, when an
extension goes over the goal value $N$, when it introduces a refinable
part, or when the \ps is saturated according to
Definition~\ref{def:saturated}, and therefore no non-trivial extension
would ever be unrefinable.

Walking along the tree, we update the values $\vec{p}$ using the same
\textsc{Update} function that we used in Algorithm~\ref{alg:simple}.
The idea is that the computation done by the recursive process on the
sequence corresponding to some path is the same as the
one done by Algorithm~\ref{alg:simple} on the same sequence.
Formally we consider
\begin{itemize}
  \item[$P1$] the set of pairs $(\aname,\vec{p}\ )$ such that  $\aname$
  is unrefinable and not saturated, $\mex(\aname)=\mu$, 
  $\asum(\aname) < N$, and such that running Algorithm~\ref{alg:simple} on
  $\aname$ computes the values $\vec{p}$;
  \item[$P2$] the set of pairs $(\aname,\vec{p}\ )$ such that the
  execution of \textsc{Enumerate}$(N,\aname^{\mu},(\infty,\ldots,\infty))$ produces
  a recursive call \textsc{Enumerate}$(N,\aname,\vec{p}\ )$.
\end{itemize}

\begin{lemma}\label{lmm:equivalence}
  The two sets $P1$ and $P2$ are equal.
\end{lemma}
\begin{proof}
  We prove this statement by induction on the length of the sequence.
  For the base case, the \ps $\aname^{\mu}$, paired
  with all $p_{j}$s set to $\infty$, is both in $P1$ and $P2$ because
  $\asum(\aname^{\mu})<N$.

  For the induction step, consider the pair $(\aname,\vec{p}\ )$ for
  which we know that $\aname$ is unrefinable, is not saturated, that
  $\mex(\aname)=\mu$ and $\asum(\aname) < N$, and that a recursive
  call \textsc{Enumerate}$(N,\aname,\vec{p}\ )$ occurs.
  
  For the extension $\aname\cup\{r\}$ the values of $\vec{p}$ do not
  change in both algorithms, therefore if $\aname$ is not saturated,
  neither is $\aname\cup\{r\}$.
  The pair $(\aname \cup \{r\},\vec{p} \ )$ is in $P1$ if and only if
  $r < p_{j}$ for $r=j \pmod{\mu}$ and $\asum(\aname)+r < N$. But these
  are exactly the same condition for the recursive call
  \textsc{Enumerate}$(N,\aname\cup \{r\},\vec{p}\ )$.

  Considering the extension $\aname\cup\{\star\}$, this is of course
  as unrefinable as $\aname$ and the sum does not change either.
  Let $\vec{q} \leftarrow$\textsc{Update}$(\ \vec{p},r)$.
  The pair $(\aname\cup\{\star\},\vec{q}\ )$ is in $P1$ if and only if
  $\aname\cup\{\star\}$ it is not saturated , and that is the exact
  same condition for the recursive call
  \textsc{Enumerate}$(N,\aname\cup \{\star\},\vec{q}\ )$ to happen.
\end{proof}

We are ready to show that, provided the appropriate input,
\textsc{Enumerate} correctly produces all the unrefinable partitions
of \(N\) with a given minimal excludant \(\mu\).

\begin{lemma}\label{lmm:enumerate}
  The recursive algorithm
  \textsc{Enumerate}$(N,\aname^{\mu},\vec{p}\ )$, where 
  $\vec{p}=(p_{0}, \ldots, p_{\mu-1})$ are all set to $\infty$,
  outputs the unrefinable \pss whose sum is $N$ with minimum
  excludant $\mu$, and without $\star$ in the last position.
\end{lemma}

\begin{proof}
  By definition, the output of the enumeration only includes \pss of $N$, not ending with $\star$. We need to prove that
  the output includes all unrefinable ones and no refinable ones.
  
  Any unrefinable \ps of $N$ with minimal excludant
  $\mu$, not ending with $\star$, can be written as
  $\aname \cup \{r\}$ where $\mex(\aname)=\mu$ and
  $\asum(\aname) = N - r < N$.
  By Lemma~\ref{lmm:equivalence}, there is a recursive call
  \textsc{Enumerate}$(N,\aname,\vec{p}\ )$ where $\vec{p}$ are the
  values computed by Algorithm~\ref{alg:simple} on $\aname$.
  By the correctness of Algorithm~\ref{alg:simple} it must be $r < p_{j}$ for
  $j = r \pmod{\mu}$ since $\aname \cup \{r\}$ is unrefinable.
  Hence the call \textsc{Enumerate}$(N,\aname,\vec{p}\ )$ outputs
  $\aname \cup \{r\}$.
  
  Now we want to show that no refinable \ps of $N$ is
  in the output.
  Consider the shortest prefix $\aname \cup \{r\}$ of any such
  sequence where $\aname$ is unrefinable and $\aname \cup \{r\}$
  is refinable.
  It still holds that $\mex(\aname)=\mu$ and that $\asum(\aname)<N$,
  therefore, by Lemma~\ref{lmm:equivalence}, there is a recursive call
  \textsc{Enumerate}$(N,\aname,\vec{p}\ )$ where $\vec{p}$ are the
  values computed by Algorithm~\ref{alg:simple} on $\aname$.
  By the correctness of Algorithm~\ref{alg:simple}, it must be $r \geq p_{j}$ for
  $j = r \pmod{\mu}$ since $\aname \cup \{r\}$ is refinable.
  Hence \textsc{Enumerate} skips $\aname \cup \{r\}$ and all its extensions.
\end{proof}

We are ready to describe the algorithm that enumerates all unrefinable
partitions of \(N\).

\begin{algorithm}
\DontPrintSemicolon
\SetKwInOut{Input}{Input}\SetKwInOut{Output}{Output\ }
\SetKw{Continue}{continue}
\SetKw{Yield}{output}
\Input{$N$}
\Output{All unrefinable partitions of $N$.}
\BlankLine
$n \leftarrow$ largest \(n\) such that \(\sum_{i=1}^{n} \leq N\)\;
\lIf{$\sum_{i=1}^{n} = N$} {\Yield $(1,2,3,\ldots,n)$}
\BlankLine
\For{$\mu$ in $\{n,n-1, \ldots, 2, 1 \}$}{%
  $\vec{p}=(p_0, \ldots, p_{\mu-1}) \leftarrow (\infty, \infty,\ldots,\infty)$\;
  \(\aname^{\mu} \leftarrow (1,2,3,\ldots,\mu-1,\star)\)\;
  \textsc{Enumerate}($N$,\(\aname^{\mu}\),$\ \vec{p}\ $)
}
\caption{\label{alg:final} \textsc{UnrefinablePartitions} (enumerate all unrefinable partitions
  of $N$)}
\end{algorithm}

\begin{theorem}
  \label{thm:enumeration}
  Algorithm~\ref{alg:final} outputs all unrefinable partitions
  of $N$ in time \mbox{$O(N) \cdot U(N)$}, where $U(N)$ is the number of
  unrefinable partitions of integers $< N$.
\end{theorem}
\begin{proof}
  The algorithm correctly outputs \(\aname^{\dagger}\) if and only if
  \(N\) is a triangular number.
  All other unrefinable partitions have, as discussed above, a minimal
  excludant \(1 \leq \mu \leq n\), where \(n\) is defined as in the
  algorithm.
  Any such partition is uniquely represented by an unrefinable \ps not
  ending with a \(\star\) and having as prefix the appropriate
  \(\aname^{\mu}\). By Lemma~\ref{lmm:enumerate} these sequences are
  produced by the calls to \textsc{Enumerate} in the main cycle.
  Hence the algorithm correctly enumerates the unrefinable partitions
  of \(N\).
  
  To discuss the runtime, observe that each of the calls to
  \textsc{Enumerate} in the main cycle makes a number of recursive call
  equal to the set $P1$ discussed in Lemma~\ref{lmm:equivalence}.
  Collecting together all the call of all values of \(\mu\), this is
  by definition the number of unrefinable and unsaturated \pss of
  a number $ < N$.
  For any unrefinable partitions of a number $\leq N$, by
  Proposition~\ref{proposition:atmostmu} there are at most $O(\sqrt{N})$ such \pss.
  Hence the process produces at most $O(\sqrt{N})$ recursive calls
  for each unrefinable partitions of a number \(<N\).
  The cost of each recursive call is dominated by the cost of the call
  to \textsc{Update}, hence $O(\sqrt{N})$, plus the possible cost to
  output an actual unrefinable partition of \(N\), when encountered,
  which is again at most $O(\sqrt{N})$.
  
  Therefore the total cost of the \(O(N)\) times the number\footnote{At the time of writing, $U(N)$ is unknown.} of all
  unrefinable partitions of integers less than \( N\).
\end{proof}

An implementation in C++ of the presented  algorithms,  made available online at~\url{https://github.com/MassimoLauria/a179009},
has been used to compute the number of unrefinable partitions of $N$, with $N$ up to 1500. The full list can 
be found \href{https://massimolauria.net/perm/unrefinable_01500.txt}{online}.
Some of the data are also available here 
in Table~\ref{tab:data} (cf.\ also~\cite[\url{https://oeis.org/A179009}]{OEIS}).

\begin{table}[h]
\centering
\begin{tabular}{|rr|rr|}
\hline
$N$ & unref.\ partitions of $N$ & $N$ & unref.\ partitions of $N$ \\
\hline
10 & 1 & 400 & 57725\\
20 & 7 & 500 & 275151\\
30 & 5 & 1000 & 84527031\\
40 & 9 & 1100 & 220124218\\ 
50 & 15 & 1200 & 559471992\\
100 & 104 & 1300 & 1383113838\\
200 & 1616 & 1400 & 3357904448\\
300 & 11801 & 1500 & 7734760269\\
  \hline
\end{tabular}
  \caption{The number of unrefinable partition for some integers up to
    1500. A full list up to 1500 is available at 
    \url{https://massimolauria.net/perm/unrefinable_01500.txt}.}
    \label{tab:data}
\end{table}

\section*{Acknowledgments}
Riccardo Aragona and Roberto Civino are members of INdAM-GNSAGA\@.
Roberto Civino is funded by the Centre of Excellence
ExEMERGE at the University of L'Aquila.
The authors are thankful to the referees for their feedback and for their valuable comments improving the quality of the manuscript.
 
\bibliographystyle{halpha-abbrv}
\bibliography{sym2n_ref}

\end{document}